\documentclass[letterpaper,10pt,conference]{ieeeconf}
\usepackage{mathtools}
\usepackage{amssymb, amsmath, bm}
\usepackage[thmmarks, amsmath]{ntheorem}
\usepackage{xcolor}
\usepackage[labelformat=simple]{subcaption}
\usepackage{hyperref}
\let\labelindent\relax
\usepackage{enumitem}
\newcommand\numberthis{\addtocounter{equation}{1}\tag{\theequation}}

\usepackage[style=ieee,dashed=false]{biblatex}
\IEEEoverridecommandlockouts                             
\title{\LARGE \bf
Prevailing against Adversarial Noncentral Disturbances: \\ Exact Recovery of Linear Systems with the $l_1$-norm Estimator}

\author{Jihun Kim and Javad Lavaei
\thanks{Jihun Kim and Javad Lavaei are with the Department of Industrial Engineering and Operations Research, University of California, Berkeley. Emails: \{jihun.kim, lavaei\}@berkeley.edu}%
\thanks{This work was supported by the U. S. Army Research Laboratory and the U. S. Army Research Office under Grant W911NF2010219, Office of Naval Research under Grant N000142412673, AFOSR, NSF, and the UC Noyce Initiative.}
}

\theoremstyle{plain} 
\theoremnumbering{arabic} 
\theoremseparator{.} 
\newtheorem{theorem}{Theorem}
\newtheorem{lemma}{Lemma}

\theoremstyle{plain} 
\theoremheaderfont{\itshape} 
\theorembodyfont{\itshape} 
\theoremnumbering{arabic} 
\theoremseparator{:} 
\newtheorem{definition}{Definition}
\newtheorem{assumption}{Assumption}

\theoremstyle{plain}
\theoremheaderfont{\itshape}  
\theorembodyfont{\normalfont} 
\theoremseparator{:}          
\newtheorem{example}{Example}
\newtheorem{remark}{Remark}

\DeclareMathOperator*{\argmin}{arg\,min}
\allowdisplaybreaks

\begin{document}

\maketitle
\thispagestyle{empty}
\pagestyle{empty}

\begin{abstract}
This paper studies the linear system identification problem in the general case where the disturbance is sub-Gaussian, correlated, and possibly adversarial. First, we consider the case with noncentral (nonzero-mean) disturbances for which the ordinary least-squares (OLS) method fails to correctly identify the system. We prove that the $l_1$-norm estimator accurately identifies the system under the condition that each disturbance has equal probabilities of being positive or negative. This condition restricts the sign of each disturbance but allows its magnitude to be arbitrary. Second, we consider the case where each disturbance is adversarial with the model that the attack times happen occasionally but the distributions of the attack values are arbitrary. We show that when the probability of having an attack at a given time is less than 0.5 and each attack spans the entire space in expectation, the $l_1$-norm estimator prevails against any adversarial noncentral disturbances and the exact recovery is achieved within a finite time. These results pave the way to effectively defend against arbitrarily large noncentral attacks in safety-critical systems.
\end{abstract}

\section{Introduction}\label{intro}
The system identification aims to identify the unknown parameters that govern an underlying dynamical system, given the history of states determined by the true parameters and the disturbances. In real-world applications, one may not be able to model the system accurately, and therefore it is vital to use the collected online data to identify the system. The identification is often followed by the control of the system, and a small identification error improves the quality of the control and affects the stability and optimality of adaptive control strategies \cite{kumar1986stochastic, ioannou2012robust}. 
In this paper, we focus on the \textit{linear system identification}, where the states are determined by $x_{t+1} = A^* x_t + w_t, ~t=0, \dots, T-1$, and the goal is to find 
the true unknown matrix $A^*$ using the state trajectory, despite the presence of disturbances $w_t$.

The challenge of this problem is that the states are correlated with each other, even when the disturbances are assumed to be independent and identically distributed (i.i.d.). 
Early papers have shown that the ordinary least-squares (OLS) estimator $\argmin_{A}\sum_{t=0}^{T-1}\|x_{t+1}-Ax_t\|_2^2$ overcomes this correlation and achieves the consistency and the convergence towards the true matrix \cite{ljung1976ols, moore1978strong, lai1982least}. Recently, instead of the asymptotic properties, several studies have focused on non-asymptotic analyses and provided an
error rate of $O\bigr(\frac{1}{\sqrt{T}}\bigr)$ between the true matrix and OLS estimates  under i.i.d. zero-mean Gaussian and sub-Gaussian disturbances after finite samples are collected \cite{simchowitz2018learningwithout, jedra2020finite, sarkar2021finite}. 
These results were extended to use the OLS estimates for the linear quadratic control with unknown matrices, in a way that 
the estimates are treated as the true matrices after a finite time, achieving $O(\sqrt{T})$ regret \cite{mania2019certainty, simchowitz2020naive}.
While OLS is indeed a statistically optimal estimator in such i.i.d. cases, little is known about the system identification when the disturbances are correlated or possibly selected adversarially based on the past information. 

Recently, in the machine learning literature, the online control problem with any bounded adversarial disturbance, called online nonstochastic control, has gained considerable attention, among which \cite{hazan2020nonstochastic, simchowitz2020improper, chen2021blackbox} leveraged the OLS method to recover the matrix and attained $O(T^{2/3})$ regret. However, this regret implies a substantial cumulative loss, leading to an impractical algorithm for real-world applications. To address more applicable situations, \cite{feng2021adv, yalcin2023exact, zhang2024exact} assumed that adversarial disturbances occur with probability less than 1, and recovered the true matrix with a non-smooth $l_2$-norm estimator defined by $\argmin_{A}\sum_{t=0}^{T-1}\|x_{t+1}-Ax_t\|_2$. In this case, practicality is achieved since the regret does not grow after the exact recovery. However, they assumed that each nonzero adversarial disturbance has a zero mean, which overlooks the possibility of systemic bias in the disruptions.

In this paper, we allow adversarial noncentral (nonzero-mean) disturbances and identify the unknown matrix within a finite time by using the $l_1$-norm estimator, given by $\argmin_{A}\sum_{t=0}^{T-1}\|x_{t+1}-Ax_t\|_1$. 
Our work presents the first result in the literature for the exact retrieval of the true matrix even in the presence of adversarial noncentral disturbances. Under the assumptions that the probability of having an attack at a given time is $p\in(0,1)$ (see Assumption \ref{null}) and each attack covers the entire space in expectation (see Assumption \ref{nondg}), we construct two scenarios of the adversarial disturbance structure: 
\begin{enumerate}[leftmargin=1.75em]
    \item symmetric disturbance around zero \textit{scaled} by a random vector (see Assumption \ref{advnon}) with $0<p<1$, 
    \item \textit{any} adversarial noncentral disturbance with $0<p<0.5$. 
\end{enumerate}
The first scenario is applicable to the domain where positive and negative disturbances occur with equal chance, each of which can be scaled arbitrarily. For example, beliefs based on behavioral bias may cause fluctuations in the stock market \cite{lo2004adaptive} or energy demand \cite{werth2021energy}, which can be explained by disturbances of arbitrary magnitude.
A more practical situation arises in the second scenario: it applies when an extremely large attack with a nonzero mean occasionally affects the system, such as natural disaster on power systems \cite{buldyrev2010cat, wang2016power, yan2020power}, unanticipated malicious cyber attacks \cite{pasqualetti2013cyber, duo2022cyber}, and others. 
In this paper, we will show that the $l_1$-norm estimator is indeed a robust estimator capable of successfully prevailing over such disturbances. 
Note that while the paper focuses on autonomous linear systems, it can be readily generalized to linear systems with inputs as well as nonlinear systems linearly parameterized by basis functions. These types of generalization can be achieved by adopting the formulation and techniques studied in \cite{yalcin2023exact,zhang2024exact}, but this will be left as future work due to space restrictions.

The paper is organized as follows. In Sections \ref{Prelim} and \ref{sec: prob formulation}, we introduce the preliminaries and formulate the problem, respectively. In Section \ref{exact recovery}, we study the two scenarios of adversarial noncentral disturbances and show that the exact recovery is achieved within finite time. In Section \ref{sec: experiments}, we present numerical experiments that support our main results.
Finally, concluding remarks are provided in Section \ref{sec: conclusion}.

\textbf{Notation.} Let $\mathbb{R}^d$ denote the set of $d$-dimensional vectors and $\mathbb{R}^{d\times d}$ denote the set of $d\times d$ matrices. For a matrix $A$, $\|A\|_2$ denotes the operator norm and $\|A\|_F$ denotes the Frobenius norm of the matrix. 
For square matrices $A$ and $B$, $A\succeq B$ means that $A-B$ is positive semidefinite. 
For a vector $x$, $\|x\|_1$ denotes the $l_1$-norm of the vector and $\|x\|_2$ denotes the $l_2$-norm of the vector. $x^T$ denotes the transpose of the vector $x$. 
For two vectors $x$ and $y$ of the same dimensions, let $x\circ y$ denote entry-wise product of the same dimensions as $x$ and $y$. 
For a scalar $z\neq 0$, $\text{sgn}(z)=1$ if $z>0$ and $\text{sgn}(z)=-1$ if $z<0$. 
Let $\mathbb{E}$ denote the expectation operator.  For the event $\mathcal{E}$, let $\mathbb{P}(\mathcal{E})$ denote the probability of the event. 
For the set $S$, let $|S|$ denote the cardinality of the set.
We use $\Theta(\cdot)$ for the big-$\Theta$ notation, and $\tilde{\Theta}(\cdot)$ for the big-$\Theta$ notation hiding logarithmic factors. Finally, let $\mathbb{S}^{d-1}$ denote the set $\{y\in\mathbb{R}^d : \|y\|_2=1\}$.

\section{Preliminaries}\label{Prelim}

In this work, we consider each disturbance vector to follow sub-Gaussian distribution. We use the definition of sub-Gaussian given in \cite{vershynin2018high}. 

\begin{definition}[sub-Gaussian scalar variables]\label{subgdef1}
    A random variable $w\in \mathbb{R}$ is called sub-Gaussian if there exists $c>0$ such that 
    \begin{equation}\label{subG}
    \mathbb{E}\Bigr[\exp\Bigr(\frac{w^2}{c^2}\Bigr)\Bigr]\leq 2.
    \end{equation}
    Its sub-Gaussian norm is denoted by $\|w\|_{\psi_2}$ and defined as 
    \begin{equation}\label{norm}
       \|w\|_{\psi_2} = \inf\left\{c>0: \mathbb{E}\Bigr[\exp\Bigr(\frac{w^2}{c^2}\Bigr)\Bigr]\leq 2\right\}. 
    \end{equation}
\end{definition}


Note that the $\psi_2$-norm satisfies all properties of norm, including positive definiteness, homogeneity, and triangle inequality. 
We have the following useful property of the moments of any sub-Gaussian variable $w\in \mathbb{R}$:
    \begin{equation}\label{moment}
        \mathbb{E}[|w|^p]^{1/p}\leq 3 \sqrt{p}\cdot \|w\|_{\psi_2}, \quad \forall p\geq 1.
    \end{equation}  
    
 We also introduce below an equivalent definition in the case where $\mathbb{E}[w]=0$.
\begin{definition}[sub-Gaussian scalar variables (MGF)]\label{mgfdef}
A random variable $w\in \mathbb{R}$ with $\mathbb{E}[w]=0$ is called sub-Gaussian if there exists $\eta_w >0 $ such that the moment generating function (MGF) of $w$ satisfies
    \begin{equation}\label{mgf}
        \mathbb{E}[\exp(\theta w)]\leq \exp(\theta^2 \eta_w^2 )
    \end{equation}
    for all $\theta\in\mathbb{R}$. For a sub-Gaussian variable that satisfies \eqref{subG} and $\mathbb{E}[w]=0$, \eqref{mgf} is satisfied with $\eta_w \leq 6\sqrt{e}\|w\|_{\psi_2}$ \cite{vershynin2018high}.
\end{definition}

Real-world systems have a limit on their actuators so they cannot accept arbitrarily large inputs. This enables us to assume that each disturbance follows sub-Gaussian variables of which the tail event rarely occurs. 
The following lemma presents Hoeffding's inequality \cite{vershynin2018high}, providing sharp bounds on the tail event for both Definitions \ref{subgdef1} and \ref{mgfdef}.

\begin{lemma}[Hoeffding's inequality]\label{hoeffding}
Suppose that a random variable $w$ satisfies \eqref{subG} with its sub-Gaussian norm defined in \eqref{norm}. Then, for all $s>0$, we have
\begin{equation}\label{ghoeff}
    \mathbb{P}(|w|\geq s) \leq 2\exp\Bigr(-\frac{s^2}{\|w\|_{\psi_2}^2}\Bigr). 
\end{equation}
In addition, if a random variable $w$ has zero mean and satisfies \eqref{mgf} with parameter $\eta_w$, then for all $s>0$, we have
    \begin{subequations}
    \begin{align}
        & \mathbb{P}(w\geq s) \leq \exp\Bigr(-\frac{s^2}{(2\eta_w)^2}\Bigr),\label{plus}
        \\ & \mathbb{P}(w\leq -s) \leq \exp\Bigr(-\frac{s^2}{(2\eta_w)^2}\Bigr).\label{minus}
    \end{align}
    \end{subequations}
\end{lemma}


To provide the analysis of high-dimensional systems, we provide the definition of multi-dimensional sub-Gaussian variables given in \cite{vershynin2018high}. 

\begin{definition}[sub-Gaussian vector variables]\label{subgvec}
    A random vector $w\in\mathbb{R}^d$ is called sub-Gaussian if for every $x\in\mathbb{R}^d$, $w^T x$ is a sub-Gaussian scalar variable defined in Definition \ref{subgdef1}. Its norm is defined as
    \begin{equation}\label{normdef}
        \|w\|_{\psi_2} = \sup_{\|x\|_2 \leq 1, x\in\mathbb{R}^d}\|w^T x\|_{\psi_2}. 
    \end{equation}
\end{definition}


Now, the following lemma provides the connection between the $\psi_2$-norm and the operator norm. 

\begin{lemma}\label{uppsi}
    For any matrix $A\in \mathbb{R}^{d_1 \times d_2}$ and any sub-Gaussian vector $x \in \mathbb{R}^{d_2}$, we have $\|Ax\|_{\psi_2} \leq \|A\|_2 \|x\|_{\psi_2}$.
\end{lemma}

\begin{proof}
Given a matrix $A$ and a vector $x$, we have
    \begin{align*}
     \|Ax\|_{\psi_2} &= \sup_{\|y\|_2 \leq 1, y\in\mathbb{R}^{d_1}} \|x^T A^T y\|_{\psi_2} \\&\leq \sup_{\|A^T y\|_2\leq \|A\|_2, y\in\mathbb{R}^{d_1}} \|x^T A^T y\|_{\psi_2} \\ &\leq \sup_{\|z\|_2\leq \|A\|_2, z\in\mathbb{R}^{d_2}} \|x^T z\|_{\psi_2}
    \\&= \sup_{\|\tilde{z}\|_2\leq 1, \tilde{z}\in\mathbb{R}^{d_2}} \Bigr\| \|A\|_2 \cdot x^T \tilde{z}\Bigr\|_{\psi_2}=\|A\|_2 \cdot \|x\|_{\psi_2},
\end{align*}
where the first inequality comes from the relationship 
\[
\|y\|_2 \leq 1 ~\Rightarrow~ \|A^T y\|_2\leq \|A\|_2 \|y\|_2 \leq \|A\|_2
\]
and the last equality is due to the homogeneity of the $\psi_2$-norm. 
\end{proof}


\section{Problem Formulation} \label{sec: prob formulation}

Consider a general linear time-invariant dynamical system with additive disturbance
\begin{equation}\label{sysd}
    x_{t+1} = A^* x_t + w_t, \quad t=0,1,\dots,T-1,
\end{equation}
where $A^*\in\mathbb{R}^{d\times d}$ is the unknown true matrix, $x_t\in\mathbb{R}^d$ is the state, and $w_t\in \mathbb{R}^d$ is a disturbance injected into the system at time $t$. Our goal is to identify the true matrix $A^*$, given the state trajectory $x_0, \dots, x_T$. We assume that $A^*$ has bounded operator norm and each $x_0, w_0, \dots, w_{T-1}$ is sub-Gaussian to prevent an unbounded growth of the system.  
We formally present the relevant assumptions 
below. 

\begin{assumption}[Operator norm]\label{stability}
    It holds that $\|A^*\|_2 <1$ (this condition is somewhat stronger than stability).
\end{assumption}

\begin{assumption}[Maximum sub-Gaussian norm]\label{maxnorm}
    Define a filtration $\mathcal{F}_t = \bm{\sigma}\{x_0, w_0, \dots, w_{t-1}\}$. There exists $\sigma_w >0$ such that $\|x_0\|_{\psi_2} \leq \sigma_w$ and 
    $\|w_t\|_{\psi_2}\leq \sigma_w$ conditioned on $\mathcal{F}_t$ for all $t\geq 0$ and $\mathcal{F}_t$.
\end{assumption}

Our major advancement over previous works is to address two main challenges simultaneously: the disturbances can be adversarial rather than independent, and each of them may also have a nonzero mean. 
This allows $w_t$ to be
\textit{adversarial noncentral disturbances}, 
meaning that an adversary can design the disturbance $w_t$ based on the previous information $\mathcal{F}_t$, and $\mathbb{E}[w_t~|~\mathcal{F}_t]$ is not necessarily zero. 

Our goal is to construct an estimator that achieves a finite-time exact recovery, meaning that the error between the true matrix and the estimate should be exactly zero after a finite time. 
However, it is well-established that the lower bound of the error is $\Theta\bigr(\frac{1}{\sqrt{T}}\bigr)$ if disturbances follow i.i.d. zero-mean Gaussian distribution \cite{simchowitz2018learningwithout}. Evidently, this lower bound $\Theta\bigr(\frac{1}{\sqrt{T}}\bigr)$ does not allow exact recovery, and thus we introduce the following further assumption also adopted by \cite{feng2021adv, yalcin2023exact, zhang2024exact}.
\begin{assumption}[Probabilistic Attack]\label{null}
    The disturbance has an attack probability $0<p<1$ conditioned on $\mathcal{F}_t$, meaning that 
    \begin{equation}
        \mathbb{P}(w_t \equiv 0~|~ \mathcal{F}_t) =1-p 
    \end{equation}
    holds for all 
    $t\geq 0$ and $\mathcal{F}_t$, where $\equiv$ means that the two sides are identically equal. We also define the attack time set as $\mathcal{K}_T = \{0\leq t \leq T-1: w_t \text{~is not identically zero} \}$.
\end{assumption}

Note that Assumption \ref{null} implies that the system is not under attack at time $t$ with a positive probability $1-p$. We further require a similar assumption presented in \cite{zhang2024exact};
When the system is under attack, no adversarial disturbance can deceive the system to the extent that the distribution of the next state fails to span the entire state space, thereby ensuring sufficient exploration of the state trajectory. Remark \ref{roleattack} provides more details on justifying such role of attacks.
\begin{assumption}[Non-degeneracy]\label{nondg}
There exists $\lambda>0$ such that for every $x\in \mathbb{R}^d$, it holds that
\begin{align}
    &\mathbb{E}[(x+w_t)(x+w_t)^T~|~ \mathcal{F}_t]\succeq \lambda^2 I_d,
\end{align}
for all $t\in\mathcal{K}_T$ and $\mathcal{F}_t$, where $I_d$ is the $d\times d$ identity matrix. 
\end{assumption}

Now, given a state trajectory $x_0,\dots, x_{T}$, we consider the following $l_1$-norm estimator at time $T$:
   \begin{equation}\label{l1est}
      \min_{A\in \mathbb{R}^{d\times d}} \sum_{t=0}^{T-1}\|x_{t+1}-Ax_t\|_1.
   \end{equation}
    We will show that $l_1$-estimator successfully overcomes adversarial noncentral distributions and achieves a finite-time exact recovery. We formally define this notion below. 
\begin{definition}[Finite-time Exact Recovery]\label{fer}
    Let $S_t$ denote a set of solutions to the estimator based on $x_0, x_1, \dots, x_t$.
    The estimator is said to achieve finite-time exact recovery if for all $\delta \in (0,1]$, there exists $t_\delta>0$ such that 
    \begin{equation}
        t\geq t_\delta  ~\Rightarrow~  S_t = \{A^*\}
    \end{equation}
    with probability at least $1-\delta$.
\end{definition}

Note that Definition \ref{fer} implies that $A^*$ should be the unique solution to the estimator to achieve the exact recovery.







\section{Main Results} \label{exact recovery}

In this section, under proper assumptions, we will show that the $l_1$-norm estimator achieves the  finite-time exact recovery defined in Definition \ref{fer}.
Let $\hat{A}_T$ denote any estimate obtained from \eqref{l1est}, which can be equivalently written as 
\begin{equation}\label{1no}
    \hat{A}_T \in \argmin_{A\in \mathbb{R}^{d\times d}} \sum_{t=0}^{T-1}\|(A^*-A)x_t+w_t\|_1
\end{equation}
due to the system dynamics \eqref{sysd}. In the next subsection, we first elucidate the exact recovery conditions of the $l_1$-norm estimator.

\subsection{Conditions for the Exact Recovery}\label{4a}

In this subsection, we provide sufficient conditions for the exact recovery. For the following theorem, 
let $w_{t}^i$ denote the $i^\text{th}$ entry of $w_t$.
Given $y\in \mathbb{R}^d$, define the random variables $z_{t}^i(y)$ as follows:
    \begin{align}\label{zpl}
z_{t}^i(y)=\begin{dcases*}
|y^T x_t|, & if $w_{t}^i=0$, \\
y^T x_t\cdot \text{sgn}(w_{t}^i), & otherwise.
\end{dcases*}
\end{align}

\begin{theorem}\label{optim}
    $A^*$ is the unique solution to the $l_1$-norm estimator \eqref{l1est} at time $T$ if
    \begin{align}
            &\sum_{t=0}^{T-1} z_{t}^{i}(y) > 0, \quad \forall y\in \mathbb{S}^{d-1}\label{znotationa}
        \end{align}
  holds for all $i\in \{1,\dots,d\}$.
\end{theorem}
\begin{proof}
The equivalent condition for $A^*$ to be the unique solution of the convex optimization problem \eqref{1no} is the existence of some $\epsilon >0$ such that
\begin{align}\label{diff1}
    \nonumber &\sum_{t=0}^{T-1} \|w_t\|_1 < \sum_{t=0}^{T-1} \|\Delta\cdot x_t + w_t\|_1, \\&\hspace{34mm}\forall \Delta\in\mathbb{R}^{d\times d}: 0<\|\Delta\|_F \leq \epsilon,
\end{align}
  since a strict local minimum in convex problems implies the unique global minimum. 
    A sufficient condition for \eqref{diff1} is to satisfy all coordinate-wise inequalities. That is, if there exist $\epsilon_1, \dots, \epsilon_d >0$ such that
    \begin{equation}\label{diff2}
    \sum_{t=0}^{T-1} |w_t^i| < \sum_{t=0}^{T-1} |\Delta_i^T x_t + w_t^i|, \quad \forall \Delta_i\in\mathbb{R}^{d} : 0<\|\Delta_i\|_2 \leq \epsilon_i
\end{equation}
for all $i\in\{1,\dots,d\}$, then the inequality \eqref{diff1} is satisfied. 
For all $i$, consider a sufficiently small $\epsilon_i >0$. Then, we have
\begin{equation*}
    |\Delta_i^T x_t + w_t^i| = (\Delta_i^T x_t + w_t^i)\cdot \text{sgn}(w_t^i) = \Delta_i^T x_t\cdot  \text{sgn}(w_t^i) + |w_t^i|
\end{equation*}
for $w_t^i \neq 0$. Substituting the above equation into the right-hand side of \eqref{diff2} yields
\begin{align*}
    \sum_{t=0}^{T-1} |w_t^i| < \sum_{\substack{t=0,\\ w_t^i\neq 0}}^{T-1} (\Delta_i^T x_t\cdot  \text{sgn}(w_t^i) + |w_t^i|) +\sum_{\substack{t=0,\\ w_t^i= 0}}^{T-1}  |\Delta_i^T x_t|,
\end{align*}
which is simplified to
\begin{align*}
    0 < \sum_{\substack{t=0,\\ w_t^i\neq 0}}^{T-1} (\Delta_i^T x_t\cdot  \text{sgn}(w_t^i)) +\sum_{\substack{t=0,\\ w_t^i= 0}}^{T-1}  |\Delta_i^T x_t|.
\end{align*}
for all $0<\|\Delta_i\|_2 \leq \epsilon_i$. 
For all $i$, dividing both sides by $\|\Delta_i\|_2 >0$ leads to
the set of inequalities in \eqref{znotationa}.
\end{proof}


To attain the exact recovery, 
it suffices to show that the random variables on the left-hand sides of \eqref{znotationa} are sufficiently positive with high probability. 
In the following two subsections, 
each will present a distinct scenario in which the $l_1$-norm estimator achieves the exact recovery despite the presence of adversarial noncentral disturbances.

\subsection{Exact Recovery under Sign-restricted Disturbances}\label{4b}

In this subsection, we will demonstrate that the exact recovery is achieved under the condition that each disturbance has equal chance of being positive or negative. 
The structure of the disturbance $w_t$ consists of $\alpha_t$ and $\beta_t$, where $\alpha_t$ is a sub-Gaussian vector symmetric around zero and $\beta_t>0$ is a scaling factor. The formal assumption is given below.

\begin{assumption}[Sign-restricted disturbances]\label{advnon}
Given the filtration $\mathcal{F}_t = \{x_0, w_0, \dots, w_{t-1}\}$, a disturbance $w_t \in \mathbb{R}^d$ at the attack times is formulated as
    \begin{equation}\label{alpbet}
        w_t = \alpha_t \circ \beta_t,
    \end{equation}
    where $\alpha_t = [\alpha_t^1, \dots, \alpha_t^d]^T$ and $\beta_t =[\beta_t^1, \dots, \beta_t^d]^T $ satisfy the following conditions:
    \begin{itemize}\setlength{\itemindent}{-0em}
       \item $\alpha_t$ is a sub-Gaussian vector variable (see Definition \ref{subgvec}) symmetric around zero conditioned on $\mathcal{F}_t$, meaning that for all $i=1,\dots, d$, we have
        \begin{equation}\label{alpha}
     \mathbb{P}(\alpha_t^i > 0~|~\mathcal{F}_t) = \mathbb{P}(\alpha_t^i < 0~|~\mathcal{F}_t).  
        \end{equation}
        \item $\beta_t$ is a scaling random vector conditioned on $\mathcal{F}_t$, where $0< \beta_t^i<\infty$ holds for all $i=1,\dots, d$.
    \end{itemize}
\end{assumption}

Note that the resulting $w_t$ is indeed sub-Gaussian conditioned on $\mathcal{F}_t$ since $\alpha_t$ is sub-Gaussian and $\beta_t$ consists of finite entries.
\begin{remark}[Necessity of symmetry assumption]
To formulate adversarial noncentral disturbances, we have introduced $\alpha_t$ that is symmetric around zero. We adopted this assumption since the system cannot defend the attack if the adversary has a full information set $\mathcal{F}_t$ and the freedom of choosing any nonzero-mean attack. For example, consider the scalar system $x_{t+1} = a^* x_t +w_t$, where the true system satisfies $0<a^*<1$ and $0<x_0<1$. Suppose that $w_t$ conditioned on $\mathcal{F}_t$ is chosen by an adversary as $-\text{sgn}(x_t)$ for all $t\geq 0$.   In such a case, one always attains 
\begin{align*}
    0<x_t<1 ~&\text{for}~ t =0,2,4,\dots, \\
 -1<x_t<0 ~&\text{for}~ t =1,3,5,\dots.
\end{align*}
    However, any valid estimation method tries to minimize the distance between $x_{t+1}$ and $\hat{a} x_t$ to find an estimate $\hat{a}$. One can thus only arrive at a negative estimate $\hat{a}$ despite the true system $a^*>0$, since we have $\text{sgn}(x_t)\neq \text{sgn}(x_{t+1})$ for all $t\geq0$. 
This occurs since the attacker can always deceive the system if the attack can be universally unidirectional based on the filtration. The symmetry assumption \eqref{alpha} on $\alpha_t$ prevents such a phenomenon and enables system identification.
\end{remark}

\vspace{-1mm}
   The disturbance structure \eqref{alpbet} includes zero-mean Gaussian variables since $\alpha_t$ can be zero-mean Gaussian and $\beta_t$ can be a constant random vector. Our formulation will later generalize this by converting a symmetric distribution into a noncentral distribution by adjusting a scaling factor $\beta_t$. 

Now, we analyze the sufficient conditions for the exact recovery shown in Theorem \ref{optim}. We first provide the following lemma given in \cite{zhang2024exact} for the upper term in \eqref{zpl}.
\begin{lemma}\label{as245}
    Let $N_T$ be the cardinality of the set $\{1\leq t\leq T-1: w_t \equiv 0, w_{t-1}\in \mathcal{K}_T \}$, which is the number of non-attack times immediately preceded by an attack. Suppose that Assumptions  \ref{maxnorm}, \ref{null}, and \ref{nondg} hold.
    Given $y\in\mathbb{S}^{d-1}$ and $i\in\{1,\dots,d\}$, we have
        \begin{align}\label{4and3}
            \sum_{\substack{t=0,\\ w_t^i= 0}}^{T-1}  |y^T x_t| \geq \frac{c\lambda^5 N_T}{\sigma_w^4}
        \end{align}
        
        \vspace{-2mm}
\noindent with probability $1-\exp\bigr(-\Theta\bigr(\frac{\lambda^4 N_T}{\sigma_w^4}\bigr)\bigr)$, where $c$ is a positive absolute constant.
\end{lemma}

Now, given $y\in \mathbb{S}^{d-1}$, we present the analysis on bounding the tail probability of the left-hand side of \eqref{znotationa}, for which we will study the sub-Gaussian norms under Assumption \ref{advnon}.
\begin{lemma}\label{fixedy}
    Suppose that Assumptions \ref{stability},  \ref{maxnorm}, \ref{null}, \ref{nondg}, and \ref{advnon} hold. 
    Define $\tau := \sigma_w / \lambda$. 
    Given $y\in \mathbb{S}^{d-1}$, $i\in\{1,\dots,d\}$, and $\delta\in(0, 1]$, the condition \eqref{znotationa}
    holds when 
    \begin{equation}\label{ffbound}
    T\geq \Theta\biggr(\frac{\max\{1,\tau^{10}\}}{p(1-p)^2 (1-\|A^*\|_2)^2}\cdot \log\biggr(\frac{1}{\delta}\biggr)\biggr)
    \end{equation}
    with probability at least $1-\delta$.
\end{lemma}

\begin{proof}
    We first analyze the sub-Gaussian parameter of the lower term in \eqref{zpl}.
   Note that the symmetry assumption \eqref{alpha} implies that the adversarial disturbance $w_t$ also satisfies
    \[
    \mathbb{P}(w_t^i > 0~|~\mathcal{F}_t) = \mathbb{P}(w_t^i < 0~|~\mathcal{F}_t)  
    \]
    for all $\mathcal{F}_t$ due to $\beta_t >0$. 
    In turn, at the attack times, we have $\mathbb{E}[y^T x_t\cdot \text{sgn}(w_t^i)] = \mathbb{E}[y^T x_t\cdot \mathbb{E}[ \text{sgn}(w_t^i)~|~\mathcal{F}_t]] = 0$. Thus, we can leverage Definition \ref{mgfdef} to derive the sub-Gaussian parameter of the relevant term. Under Assumption \ref{null}, the state equation is written as \[x_t = (A^*)^tx_0 + \sum_{k\in  [0,t-1]\cap \mathcal{K}_T } (A^*)^{t-1-k} w_k\] due to the system dynamics \eqref{sysd}. 
    Define another filtration \[\mathcal{F}^i = \bm{\sigma}\{\text{sgn}(w_t^i): 0\leq t\leq T-1, w_t^i \neq 0\}.\] Then, for all $\lambda\in\mathbb{R}$ and for all $s\in [0,T-2]\cap \mathcal{K}_T$, we have
\begin{align*}
    &\mathbb{E}\Bigr[\exp\bigr(\lambda \sum_{\substack{t=s+1,\\w_t^i\neq 0 }}^{T-1}  y^T(A^*)^{t-1-s} w_{s}  \cdot  \text{sgn}(w_t^i) \bigr)~\Bigr|~\mathcal{F}_{s}, \mathcal{F}^i\Bigr]\\&\hspace{29mm}\leq \exp\biggr(\lambda^2 \cdot \Theta\Bigr(\frac{\sigma_w}{1-\|A^*\|_2}\Bigr)^2\biggr) \numberthis\label{convert}
\end{align*}
due to
\begin{align*}
&\Biggr\|\sum_{\substack{t=s+1,\\w_t^i\neq 0 }}^{T-1}  y^T(A^*)^{t-1-s} w_{s}  \cdot  \text{sgn}(w_t^i)  \Biggr\|_{\psi_2} 
\\&\hspace{10mm}\leq\sum_{\substack{t=s+1,\\w_t^i\neq 0 }}^{T-1} \bigr\|  (A^*)^{t-1-s} w_{s}  \cdot  \text{sgn}(w_t^i) \bigr\|_{\psi_2} \\&\hspace{10mm}\leq \sum_{\substack{t=s+1,\\w_t^i\neq 0 }}^{T-1}\|  A^*  \|_{2}^{t-1-s} \| w_{s}\|_{\psi_2} \leq \frac{\sigma_w}{1-\|A^*\|_2} ,\numberthis\label{befconvert}
    \end{align*}
    conditioned on $\mathcal{F}_s$ and $\mathcal{F}^i$.
     The first inequality comes from Lemma \ref{uppsi}, the triangle inequality, and $\|y\|_2=1$. The second inequality is due to Lemma \ref{uppsi} and the fact that alternating the sign does not affect the sub-Gaussian norm since the definition \eqref{norm} involves the squared variable. From \eqref{befconvert}, the equivalence of Definitions \ref{subgdef1} and \ref{mgfdef} yields \eqref{convert}. 
     
By repeatedly selecting from the largest to the smallest element for $s\in[0,T-2]\cap \mathcal{K}_T$,  
applying the tower rule conditioning on $(\mathcal{F}_s, \mathcal{F}^i)$, and leveraging \eqref{convert}, we obtain 
\begin{align*}
        &\mathbb{E}\Biggr[\exp \biggr(\lambda \biggr[\sum_{\substack{t=0,\\ w_t^i\neq 0}}^{T-1} y^T x_t\cdot  \text{sgn}(w_t^i)\biggr]\biggr)\Biggr] 
        \\ &= \mathbb{E}\Biggr[\exp \biggr(\lambda \biggr[   \sum_{\substack{t=0,\\ w_t^i\neq 0}}^{T-1}\Bigr[y^T(A^*)^t x_0 \cdot \text{sgn}(w_t^i)\Bigr] \\&\hspace{6mm}+ \sum_{k\in[0,T-2]\cap \mathcal{K}_T} \sum_{\substack{t=k+1,\\w_t^i\neq 0 }}^{T-1}  y^T(A^*)^{t-1-k} w_k  \cdot  \text{sgn}(w_t^i) \Bigr]\biggr]\biggr)\Biggr] \\ 
        &\leq  \cdots  \leq \exp\biggr(\lambda^2(1+|\mathcal{K}_T|) \cdot \Theta\Bigr(\frac{\sigma_w}{1-\|A^*\|_2}\Bigr)^2\biggr).\numberthis\label{subgparam} 
    \end{align*}
    Since the term of interest has zero mean, we can use the sub-Gaussian parameter given in \eqref{subgparam} and apply the inequality \eqref{minus} given in Lemma \ref{hoeffding} to derive
    \begin{align*}
&\mathbb{P}\biggr(\sum_{\substack{t=0,\\ w_t^i\neq 0}}^{T-1} y^T x_t\cdot  \text{sgn}(w_t^i) > -\frac{c\lambda^5 N_T}{2\sigma_w^4}\biggr) \\
& \hspace{10mm}\geq 1-\exp\Biggr( - \Theta\biggr(\frac{(\frac{\lambda^5 N_T}{\sigma_w^4})^2}{(1+|\mathcal{K}_T|)\cdot (\frac{\sigma_w}{1-\|A^*\|_2})^2}\biggr)\Biggr) \\& \hspace{10mm}= 1-\exp\Biggr( - \Theta\biggr(\frac{N_T^2 \cdot (1-\|A^*\|_2)^2}{(1+|\mathcal{K}_T|)\cdot \tau^{10} }\biggr)\Biggr) \numberthis\label{imp2}
    \end{align*}
    where $c$ and $N_T$ are the quantities in Lemma \ref{as245}. 
    Constructing the union bound using Lemma \ref{as245} and \eqref{imp2}, we obtain
    \begin{align*}
&\mathbb{P}\biggr(\sum_{\substack{t=0,\\ w_t^i\neq 0}}^{T-1} y^T x_t\cdot  \text{sgn}(w_t^i) + \sum_{\substack{t=0,\\w_t^i=0}}^{T-1}|y^T x_t| > \frac{c\lambda^5 N_T}{2\sigma_w^4}\biggr)\numberthis\label{25} \\&\hspace{5mm}\geq 1-\exp\Biggr(-\Theta\biggr(\frac{\lambda^4 N_T}{\sigma_w^4}\biggr)\Biggr) \\&\hspace{20mm}-\exp\Biggr( - \Theta\biggr(\frac{N_T^2 \cdot (1-\|A^*\|_2)^2}{(1+|\mathcal{K}_T|)\cdot \tau^{10} }\biggr)\Biggr) \numberthis\label{finalbound}
    \end{align*}
    Finally, due to Assumption \ref{null}, the probabilistic attack model implies that we have 
    \begin{equation}\label{imp3}
        N_T \geq \Theta(p(1-p)T), \quad 1+|\mathcal{K}_T| \leq \Theta(pT)
    \end{equation}
    with probability at least $1-\exp(-\Theta(p(1-p)T))$, considering the expectations of each quantity. The union bound \eqref{finalbound} is then lower-bounded by $1-\delta$ if we have
    \begin{align*}
        &T\geq \Theta\biggr(\max\Bigr\{ \frac{1}{p(1-p)}, \frac{\tau^4}{p(1-p)}, \\&\hspace{25mm}\frac{\tau^{10}}{p(1-p)^2 (1-\|A^*\|_2)^2}\Bigr\}\cdot \log\Bigr(\frac{1}{\delta}\Bigr)\biggr),
    \end{align*}
    which is equivalent to \eqref{ffbound} since $1-p$ and $1-\|A^*\|_2$ are less than $1$. 
\end{proof}


To satisfy the condition \eqref{znotationa} not only for a single $(y,i)$ pair but for all $y\in\mathbb{S}^{d-1}$ and $i\in\{1,\dots,d\}$, we leverage the following lemma presented in \cite{vershynin2010intro}.
\begin{lemma}[Covering number of the sphere]\label{covering}
    For $\epsilon>0$, consider a subset $\mathcal{N}_\epsilon$ of $\mathbb{S}^{d-1}$, such that for all    
    $y\in\mathbb{S}^{d-1}$, there exists some point $\tilde{y}\in\mathcal{N}_\epsilon$ satisfying $\|y-\tilde{y}\|_2 \leq \epsilon$. The minimal cardinality of such a subset is called the covering number of the sphere and is upper-bounded by $(1+\frac{2}{\epsilon})^d$.
\end{lemma}

Now, we present one of our main theorems to satisfy sufficient conditions for the exact recovery. 
\begin{theorem}\label{time}
    Suppose that Assumptions \ref{stability},  \ref{maxnorm}, \ref{null}, \ref{nondg}, and \ref{advnon} hold. Given $\delta\in (0,1],$
    the $l_1$-norm estimator defined in \eqref{l1est} achieves the exact recovery with probability at least $1-\delta$ when
    \begin{align}\label{timeeq}
    \nonumber &T\geq  \Theta\biggr(\frac{d\cdot \max\{1,\tau^{10}\}}{p(1-p)^2 (1-\|A^*\|_2)^2}\\&\hspace{13mm}\times \max\biggr\{1, \log\biggr(\frac{d\cdot \tau}{\delta(1-p)(1-\|A^*\|_2) }\biggr)\biggr\}\biggr).
    \end{align}
\end{theorem}



\begin{proof}
We first show that  the condition \eqref{znotationa} holds with high probability, given $i\in\{1,\dots,d\}$. 
Let \[\epsilon:= \Theta\biggr(\frac{(1-p)(1-\|A^*\|_2)}{\tau^5 \sqrt{\log(1/\delta)}}\biggr).\]
For any $y,y'\in\mathbb{S}^{d-1}$ such that $\|y-y'\|_2\leq \epsilon$, we have
\begin{align*}
    \sum_{t=0}^{T-1} z_t^i& (y) - \sum_{t=0}^{T-1} z_t^i (y') \geq -\sum_{t=0}^{T-1} |(y-y')^T x_t| \\&\geq -\|y-y'\|_2 \sum_{t=0}^{T-1} \|x_t\|_2 \geq -\epsilon\sum_{t=0}^{T-1} \|x_t\|_2 \\&\geq-\epsilon \sum_{t=0}^{T-1} \Bigr[\|A^*\|_2^t \cdot \|x_0\|_2 + \sum_{k\in\mathcal{K}_T}\|A^*\|_2^{t-1-k} \|w_k\|_2 \Bigr] \\&\geq -\frac{\epsilon}{1-\|A^*\|_2}\Bigr[\|x_0\|_2 + \sum_{k\in\mathcal{K}_T}\|w_k\|_2\Bigr]\numberthis\label{againsub}
\end{align*}
where the first inequality is due to the triangle inequality. Due to homogeneity and the triangle inequality, the sub-Gaussian norm (see \eqref{norm}) of the term in \eqref{againsub} is bounded by $\frac{\epsilon(1+|\mathcal{K}_T|)\sigma_w}{1-\|A^*\|_2}$. We leverage the inequality \eqref{ghoeff} in Lemma \ref{hoeffding} to obtain
\begin{align*}
    &\mathbb{P}\Bigr(\sum_{t=0}^{T-1} z_t^i(y) - \sum_{t=0}^{T-1} z_t^i (y') > -\frac{c\lambda^5 N_T}{4\sigma_w^4}\Bigr) \\&\geq \mathbb{P}\Bigr(-\frac{\epsilon}{1-\|A^*\|_2}\Bigr[\|x_0\|_2 + \sum_{k\in\mathcal{K}_T}\|w_k\|_2\Bigr]> -\frac{c\lambda^5 N_T}{4\sigma_w^4}\Bigr) \\&\geq 1-2\exp\biggr(-\Theta\Bigr(\frac{\lambda^{10} N_T^2 (1-\|A^*\|_2)^2}{\epsilon^2 (1+|\mathcal{K}_T|)^2\sigma_w^{10})}\Bigr)\biggr) \\&= 1-2\exp\biggr(-\Theta\Bigr(\frac{(1-p)^2 (1-\|A^*\|_2)^2}{\epsilon^2 \tau^{10}}\Bigr)\biggr) \geq 1-\frac{\delta}{2},
\end{align*}
where the equality is by considering \eqref{imp3}, provided that $T\geq \Theta(\frac{1}{p(1-p)}\log(\frac{1}{\delta}))$.  Then, by Lemma \ref{covering}, if we have a subset with $(1+\frac{2}{\epsilon})^d$ number of points $\{y_j: j=1,2,\dots\}$ in $\mathbb{S}^{d-1}$ that satisfy $\sum_{t=0}^{T-1} z_t^i(y_j) > \frac{c\lambda^5 N_T}{2\sigma_w^4}$ (see \eqref{25}) for all $j$ with probability at least $1-\frac{\delta}{2}$, then we will have 
\begin{align}\label{imbound}
    &\sum_{t=0}^{T-1} z_t^i(y)  > \frac{c\lambda^5 N_T}{4\sigma_w^4}>0, \quad \forall y\in\mathbb{S}^{d-1}
\end{align}
with probability at least $1-\delta$. Thus, it suffices to replace $\delta$ in \eqref{ffbound} with $\frac{\delta}{2(1+2/\epsilon)^d}$ to achieve \eqref{imbound}. We have
\begin{align*}
    &\log\biggr(\frac{2(1+2/\epsilon)^d} {\delta}\biggr) = \Theta\biggr(d\log\biggr(1+\frac{2}{\epsilon}\biggr)+\log\biggr(\frac{1}{\delta}\biggr)\biggr)\\&=\Theta\biggr(d\log\biggr(1+\frac{2\tau^5 \sqrt{\log(1/\delta)}}{(1-p)(1-\|A^*\|_2)}\biggr) + \log\biggr(\frac{1}{\delta}\biggr)\biggr)\\&\leq \Theta\biggr(\max\biggr\{d+\log\biggr(\frac{1}{\delta}\biggr), d\log\biggr(\frac{2\tau^5 \sqrt{\log(1/\delta)}}{\delta(1-p)(1-\|A^*\|_2)}\biggr)\biggr\}\biggr) \\&=\Theta\biggr(\max\biggr\{d, d\log\biggr(\frac{\tau }{\delta(1-p)(1-\|A^*\|_2)}\biggr)\biggr\}\biggr), \numberthis\label{maybefinal}
\end{align*}
The last equality is because the first case in the max argument arises when 
\[
\frac{2\tau^5 \sqrt{\log(1/\delta)}}{(1-p)(1-\|A^*\|_2)} \leq 1,
\]
which implies $\log(\frac{1}{\delta})\leq \Theta(\frac{1}{\tau^{10}})$ since $1-p$ and $1-\|A^*\|_2$ are less than 1. Note that Assumptions \ref{maxnorm} and \ref{nondg} result in
\begin{align*}
    \lambda^2 d &\leq \text{tr}(\mathbb{E}[w_tw_t^T ~|~ \mathcal{F}_t]) = \mathbb{E}[\text{tr}(w_tw_t^T) ~|~ \mathcal{F}_t]  \\&= \mathbb{E}[\|w_t\|^2_2 ~|~ \mathcal{F}_t]  \leq 18 \bigr\|\|w_t\|_2\bigr\|_{\psi_2} \leq 18 \bigr\|w_t\bigr\|_{\psi_2} \leq 18\sigma_w^2
\end{align*}
conditioned on any $\mathcal{F}_t$, where the second inequality is due to \eqref{moment} and the third inequality comes from Definition \ref{subgvec}. Thus, 
$\frac{1}{\tau^{10}}$ is bounded by the quantity $(\frac{18}{d})^5$, which yields $d+\log(\frac{1}{\delta}) = \Theta(d)$. Furthermore, the second case in the max argument arises from $\frac{1}{\delta}$ dominating the term $\sqrt{\log(1/\delta)}$. To conclude, we achieve \eqref{imbound} with probability at least $1-\delta$ after
\begin{align*}
    &T\geq  \Theta\biggr(\frac{d\cdot \max\{1,\tau^{10}\}}{p(1-p)^2 (1-\|A^*\|_2)^2}\\&\hspace{13mm}\times \max\biggr\{1, \log\biggr(\frac{\tau}{\delta(1-p)(1-\|A^*\|_2) }\biggr)\biggr\}\biggr).\numberthis\label{please}
\end{align*}
Letting \eqref{imbound} (or equivalently, \eqref{znotationa}) hold with probability at least $1-\frac{\delta}{d}$, and thus replacing $\delta$ in \eqref{please} with $\frac{\delta}{d}$, suffices to guarantee that 
the condition \eqref{znotationa} holds for all $i\in\{1,\dots,d\}$ with probability at least $1-\delta$. By Theorem \ref{optim}, the exact recovery is achieved with the desired probability.
\end{proof}

\begin{remark}[Role of attacks]\label{roleattack}
    In Theorem \ref{time}, the required time for exact recovery is $\tilde{\Theta}\bigr(\frac{d}{p(1-p)^2}\bigr)$, which is minimized when the attack probability $p$ is 1/3. This stems from the role of attacks in two orthogonal ways. First, the estimator requires a longer time as $p\to 1$ since the attacker tries to deceive the system. However, since the attack is also a source of excitement due to Assumption \ref{nondg}, the estimator finds it difficult to learn the system through the attacks as $p\to 0$.
    For example, if $x_0, w_0, \dots, w_{T-1}$ are all identically zero, any matrix $A$ serves as a solution to the estimator, but our goal is to guarantee that the true matrix is the \textit{unique} optimizer.
\end{remark}

\subsection{Exact Recovery under Arbitrary Disturbances}\label{4c}

In the previous subsection, we have discussed the exact recovery under Assumption \ref{advnon}, which requires an equal probability for the sign of the disturbance. In real-world applications, requiring exactly the same positive and negative probabilities may be challenging. 
We will now present a more practical situation \textit{without} considering Assumption \ref{advnon}:
if the attack probability is less than 0.5, it turns out that the $l_1$-norm estimator overcomes arbitrary (possibly adversarial) noncentral disturbance and
the exact recovery is achieved. 
This includes the case where the arbitrary attacks are \textit{always} positive (or negative), representing the worst-case scenario in which the system may easily be deceived towards a unidirectional bias. We formally state the theorem below.

\begin{theorem}\label{pless5}
    Suppose that Assumptions \ref{stability}, \ref{maxnorm}, and \ref{nondg} hold, and Assumption \ref{null} holds with $0<p<0.5$. Given $\delta \in (0,1]$, the $l_1$-norm estimator defined in \eqref{l1est} achieves the exact recovery with probability at least $1-\delta$ when
    \begin{align}\label{timeeq3}
    \nonumber &T\geq  \Theta\biggr(\frac{d\cdot \max\{1,\tau^{10}\}}{p(1-2p)^2 (1-\|A^*\|_2)^2}\\&\hspace{11mm}\times \max\biggr\{1, \log\biggr(\frac{d\cdot \tau}{\delta(1-2p)(1-\|A^*\|_2) }\biggr)\biggr\}\biggr).
    \end{align}
\end{theorem}


\begin{proof}
  We first focus on a restricted class of problems where the distribution of $w_t$ can be constructed via two types of attack vectors $v_t$ and $\tilde v_t$ as:
\begin{align}\label{vtvt}
&w_t \sim I_A \cdot [I_B \cdot v_t+ (1-I_B)\cdot \tilde{v}_t],
  \\ \nonumber &I_A \sim \text{Bernoulli}(2p), \quad I_B \sim \text{Bernoulli}(0.5),\\\nonumber & I_A \perp\!\!\!\perp I_B, \quad (I_A, I_B)\perp\!\!\!\perp v_t, \quad (I_A, I_B)\perp\!\!\!\perp \tilde v_t
\end{align}
 conditioned on $\mathcal{F}_t$, where we use $\perp\!\!\!\perp$ to denote independence. Due to $I_A$, the system is not under attack with probability $1-2p$. 
The roles of the two attacks $v_t$ and $\tilde v_t$ are distinct in the sense that the attacker can select any arbitrary distribution \textit{satisfying Assumption \ref{nondg}} for 
$v_t$, whereas the distribution of $\tilde v_t$ is strategically designed in response to $v_t$ to guarantee that the aggregate distribution of $w_t$ \textit{satisfies Assumption \ref{advnon}} at the attack times. The equation \eqref{alpha} involving $\alpha_t$ is achievable since 
the arbitrary distribution $v_t$ can be neutralized by $\tilde v_t$ to satisfy the symmetry assumption, given that $I_B$ ensures equal probabilities of occurrence for both attacks.

In this scenario, 
Assumptions \ref{null}, \ref{nondg}, and \ref{advnon} are satisfied with the aggregate attack probability $0<2p<1$. Under Assumptions \ref{stability} and \ref{maxnorm}, one can now apply Theorem \ref{time} to obtain \eqref{timeeq3}, where $1-p$ in \eqref{timeeq} is replaced with $1-2p$. This results from the modification of \eqref{imp3} as
\begin{align}\label{ntfinal}
    N_T\geq \Theta(p(1-2p)T), \quad 1+|\mathcal{K}_T| \leq \Theta(2pT)
\end{align}
with high probability, since the lower bound of $N_T$ defined in Lemma \ref{as245} accounts for the product of the probabilities of no attack ($1-2p$) and an attack satisfying Assumption \ref{nondg} (in this case, $v_t$ occurs with probability $p$). The upper bound of $1+|\mathcal{K}_T|$ should consider the total attack probability $2p$. 

Now, suppose that $\tilde v_t$ in \eqref{vtvt} is designed as $\tilde \alpha_t \circ \tilde \beta_t$, with $\tilde \alpha_t$ having finite entries.
Then, let $\tilde\beta_t \to 0^+$, meaning that all coordinates $\tilde \beta_t^i$ can be arbitrarily small but remain positive, and thus $\tilde v_t^i$ will maintain their signs with arbitrarily small magnitudes. This ensures that Assumption \ref{advnon} is not violated and \eqref{imbound} still holds after the recovery time \eqref{timeeq3}. 
Taking the limit $\tilde\beta_t \to 0^+$ can be extended to the case where $\tilde\beta_t=0$, at which point the distribution \eqref{vtvt} collapses to $w_t\sim I_A\cdot I_B\cdot v_t$, with $0<p<0.5$ and $v_t$ satisfying Assumption \ref{nondg}. This is consistent with the assumptions for $w_t$ given by this theorem, where the inequality \eqref{imbound} is  revised to a non-strict form as
\begin{align}\label{finafafaf}
    \nonumber &\sum_{t=0}^{T-1} z_t^i(y) =\sum_{\substack{t=0,\\ w_t^i\neq 0}}^{T-1} (y^T x_t\cdot  \text{sgn}(w_t^i)) +\sum_{\substack{t=0,\\ w_t^i= 0}}^{T-1}  |y^T x_t|\\ &\hspace{32mm} \geq \frac{c\lambda^5 N_T}{4\sigma_w^4}>0, \quad \forall y\in\mathbb{S}^{d-1},
\end{align}
since $y^T x_t$ is a continuous function of $w_0, \dots, w_{t-1}$ and $y^T x_t\cdot  \text{sgn}(w_t^i)$ either increases to or remains at $|y^Tx_t|$ when $\tilde\beta_t^i\to 0^+$ is converted to $\tilde\beta_t^i = 0$ (see \eqref{zpl}). Note that the last inequality of \eqref{finafafaf} remains strict, ensuring that
the condition \eqref{znotationa} is satisfied with the exact recovery time \eqref{timeeq3}.
\end{proof}
\begin{remark}
     In Theorem \ref{pless5},  we highlight that the $l_1$-norm estimator withstands arbitrary adversarial noncentral disturbances  with the exact recovery time of $\tilde{\Theta}\bigr(\frac{d}{p(1-2p)^2}\bigr)$, provided that the attack probability is restricted to $0<p<0.5$.
\end{remark}



\section{Numerical Experiments} \label{sec: experiments}

To be able to effectively demonstrate the results of this paper, we will provide two examples in this section. 
\begin{example}\label{support2}
This example shows that under Assumption \ref{advnon}, the $l_1$-norm estimator successfully recovers the true matrix, which supports Theorem \ref{time}. We generate a random matrix whose operator norm is 0.6. 
Whenever the attack times happen with some probability $p<1$,  
the adversary selects each coordinate $w_t^i$ of the disturbance $w_t$  to be $-\text{sgn}(x_t^i) \cdot \gamma$, where $\gamma$ follows $\text{Uniform}[-3,-1]$ with probability 0.5 and $\text{Uniform}[10,20]$ with probability 0.5. The disturbances satisfy the symmetry assumption given in \eqref{alpha}, but indeed have a nonzero mean.


\begin{figure}[t]
     \centering
     \begin{subfigure}[b]{0.235\textwidth}
         \centering
         \includegraphics[width=\textwidth,height=80pt]{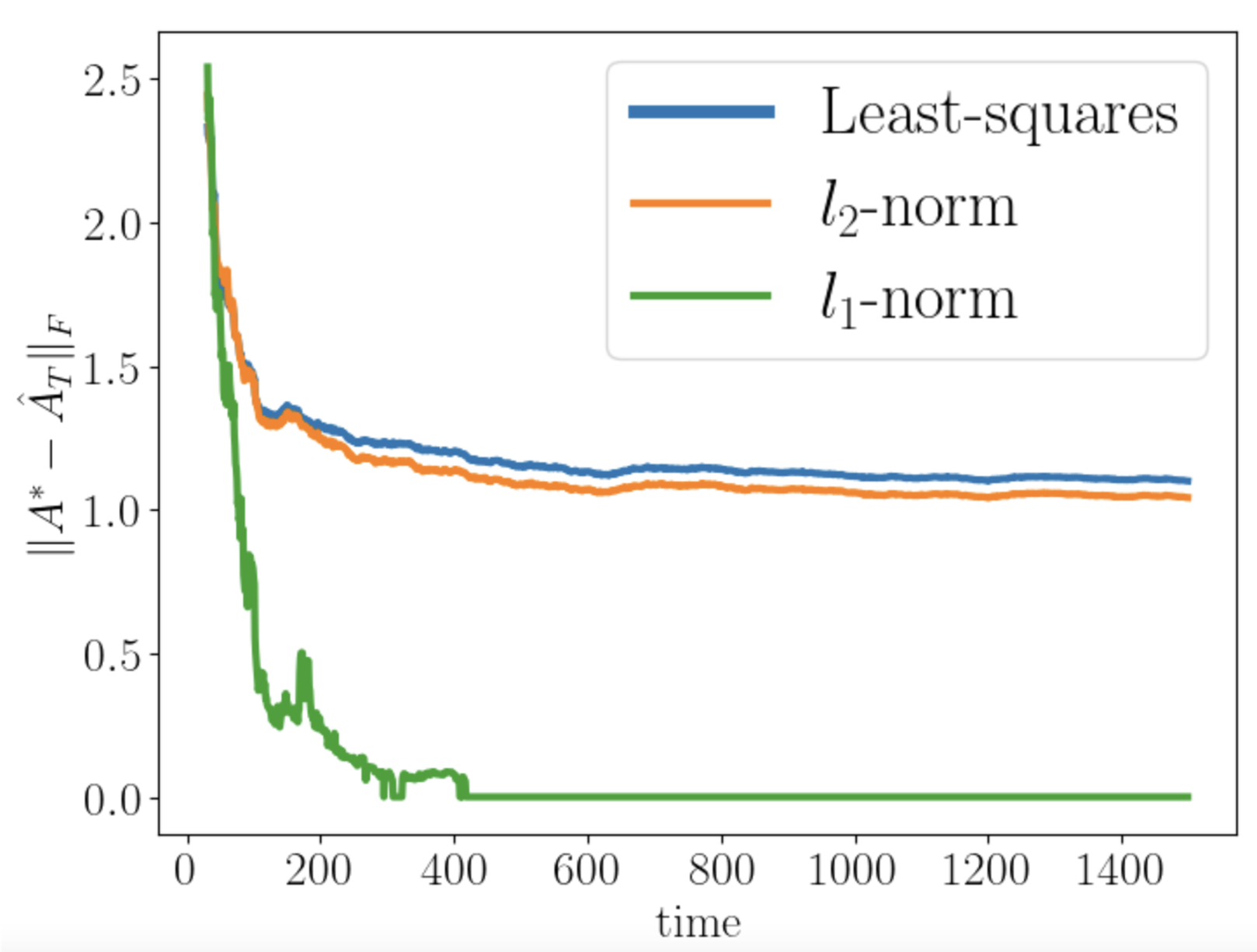}
         \caption{$p=0.7, d=10$}
         \label{dim10p30}
     \end{subfigure}
     \begin{subfigure}[b]{0.235\textwidth}
         \centering
         \includegraphics[width=\textwidth,height=80pt]{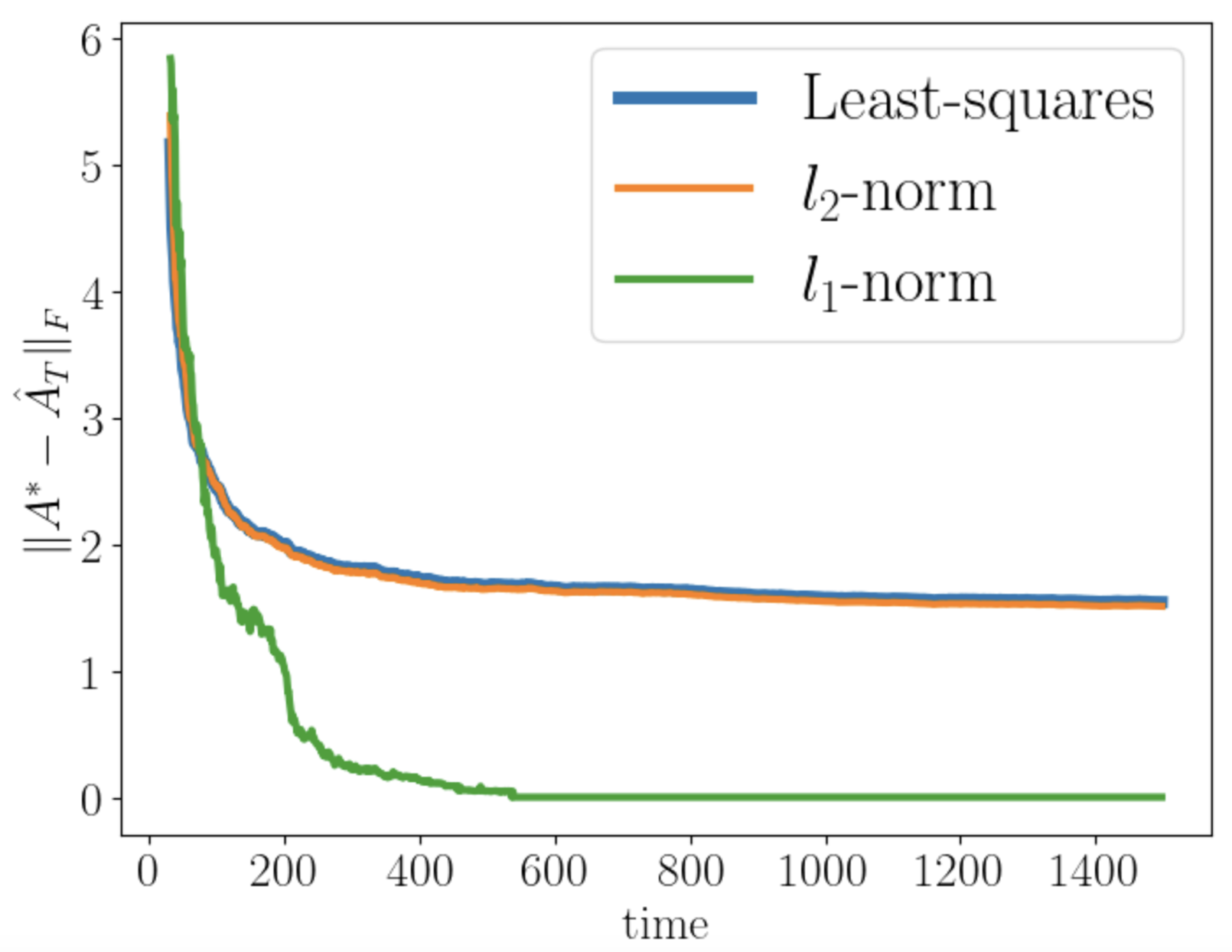}
         \caption{$p=0.7, d=20$}
         \label{dim20p30}
     \end{subfigure}
     \begin{subfigure}[b]{0.235\textwidth}
         \centering
         \includegraphics[width=\textwidth,height=80pt]{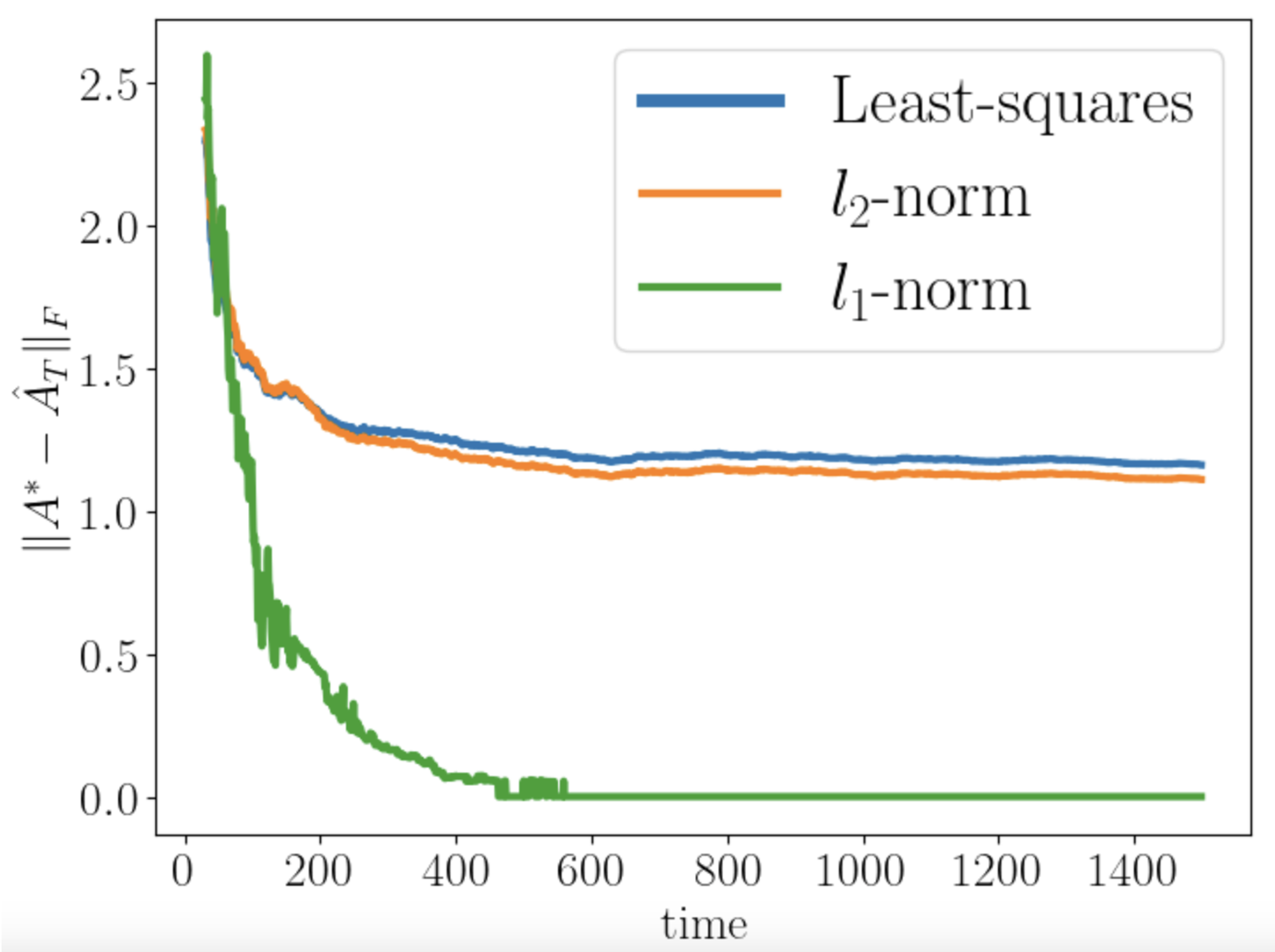}
         \caption{$p=0.75, d=10$}
         \label{dim10p25}
     \end{subfigure}
     \begin{subfigure}[b]{0.235\textwidth}
         \centering
         \includegraphics[width=\textwidth,height=80pt]{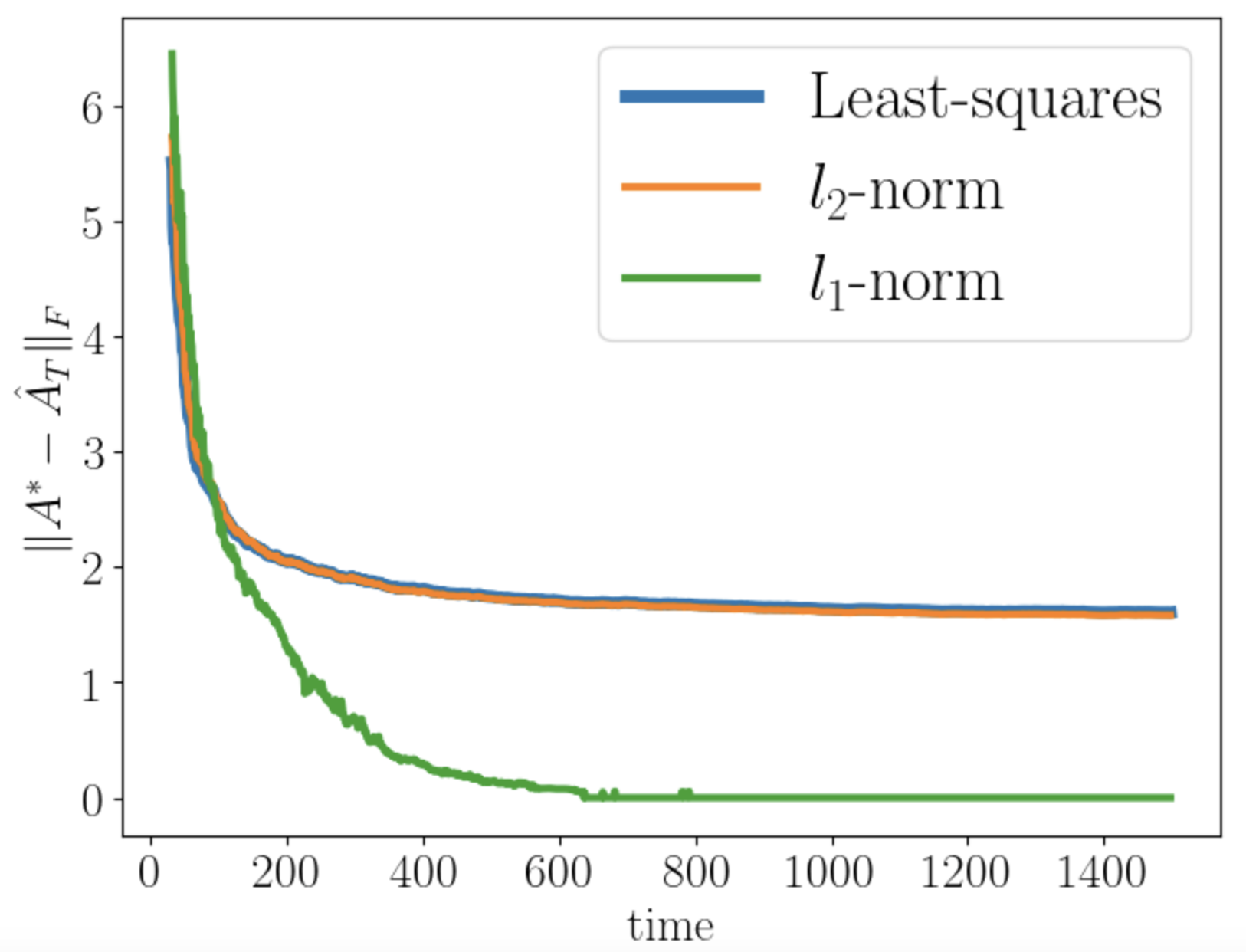}
         \caption{$p=0.75, d=20$}
         \label{dim20p25}
     \end{subfigure}
     \begin{subfigure}[b]{0.235\textwidth}
         \centering
         \includegraphics[width=\textwidth,height=80pt]{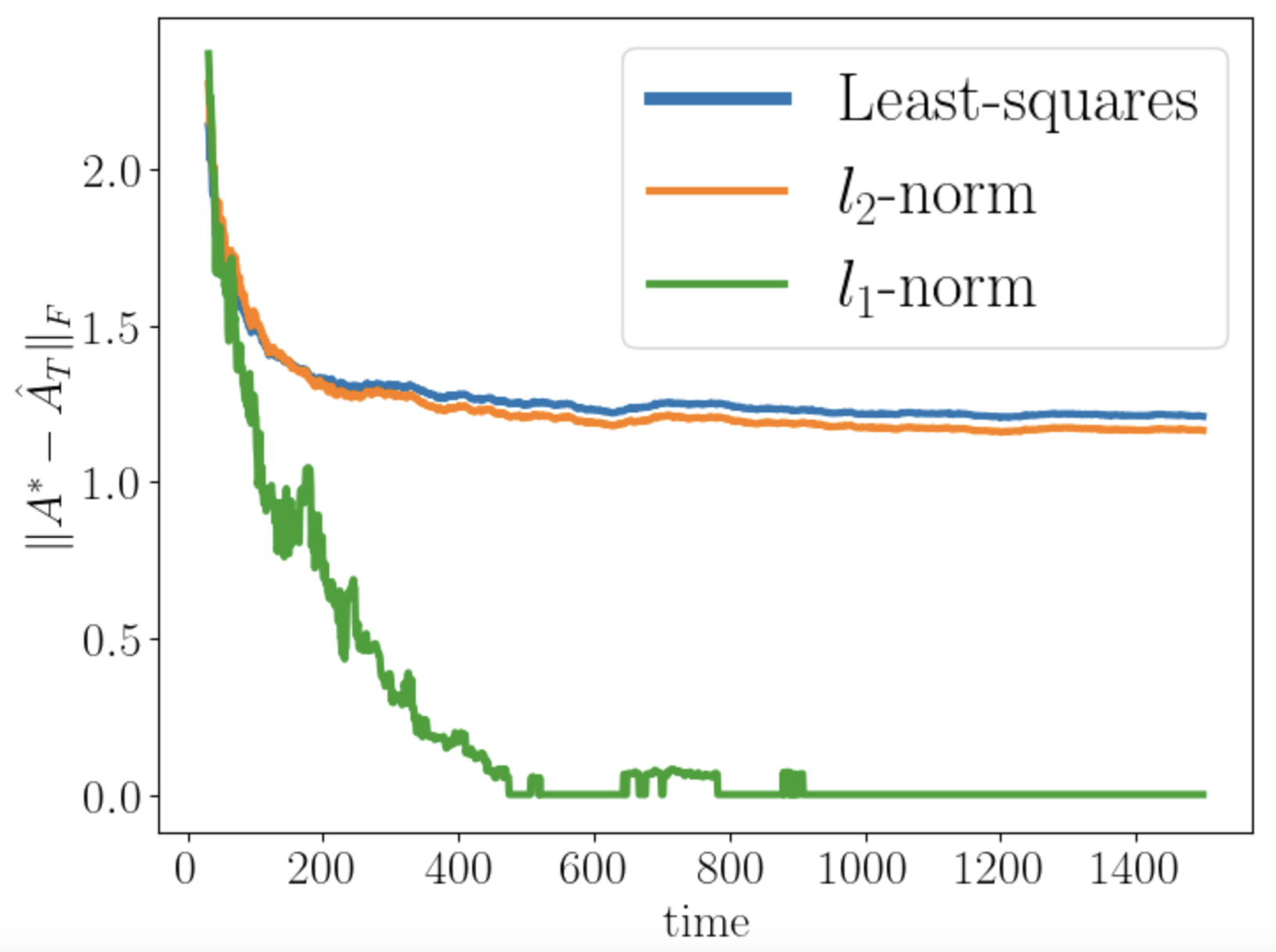}
         \caption{$p=0.8, d=10$}
         \label{dim10p20}
     \end{subfigure}
     \begin{subfigure}[b]{0.235\textwidth}
         \centering
         \includegraphics[width=\textwidth,height=80pt]{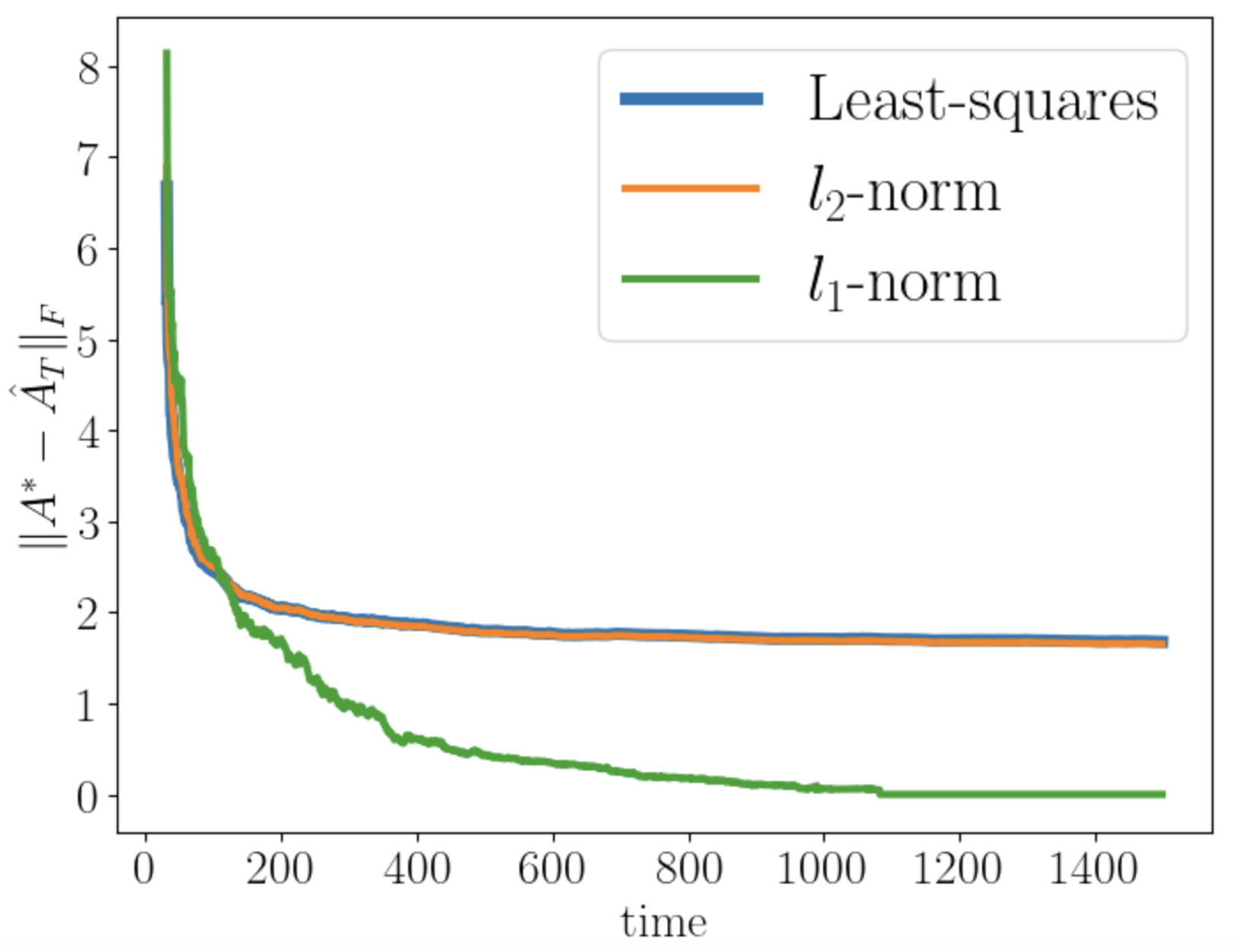}
         \caption{$p=0.8, d=20$}
         \label{dim20p20}
     \end{subfigure}
        \caption{$l_1$-norm estimator vs. other estimators in Example \ref{support2}.}
        \label{ex1}
\end{figure}

In Figure \ref{ex1}, we show the efficacy of the $l_1$-norm estimator compared to the OLS and the $l_2$-norm estimator brought up in Section \ref{intro}. All three estimators involve \textit{convex optimization}, and thus can be efficiently solved by standard optimization solvers. We report the error based on the Frobenius norm of the difference between the estimates and the true matrix. We present various scenarios based on the attack probability $p$ and the dimension $d$. 
While the OLS and $l_2$-norm estimator show a plateau in the error and fail to identify the correct matrix, the $l_1$-norm estimator achieves the exact recovery for all scenarios.
The figures with the same dimension show that as $p$ increases towards 1, indicating more frequent attacks, the exact recovery time also increases. Similarly, the figures with the same $p$ demonstrate that a larger dimension leads to a slower exact recovery. These findings illustrate the recovery time of $\tilde{\Theta}\bigr(\frac{d}{p(1-p)^2}\bigr)$ stated in Theorem \ref{time}.
\end{example}
\begin{figure}[t]
     \centering
     \begin{subfigure}[b]{0.235\textwidth}
         \centering
         \includegraphics[width=\textwidth,height=80pt]{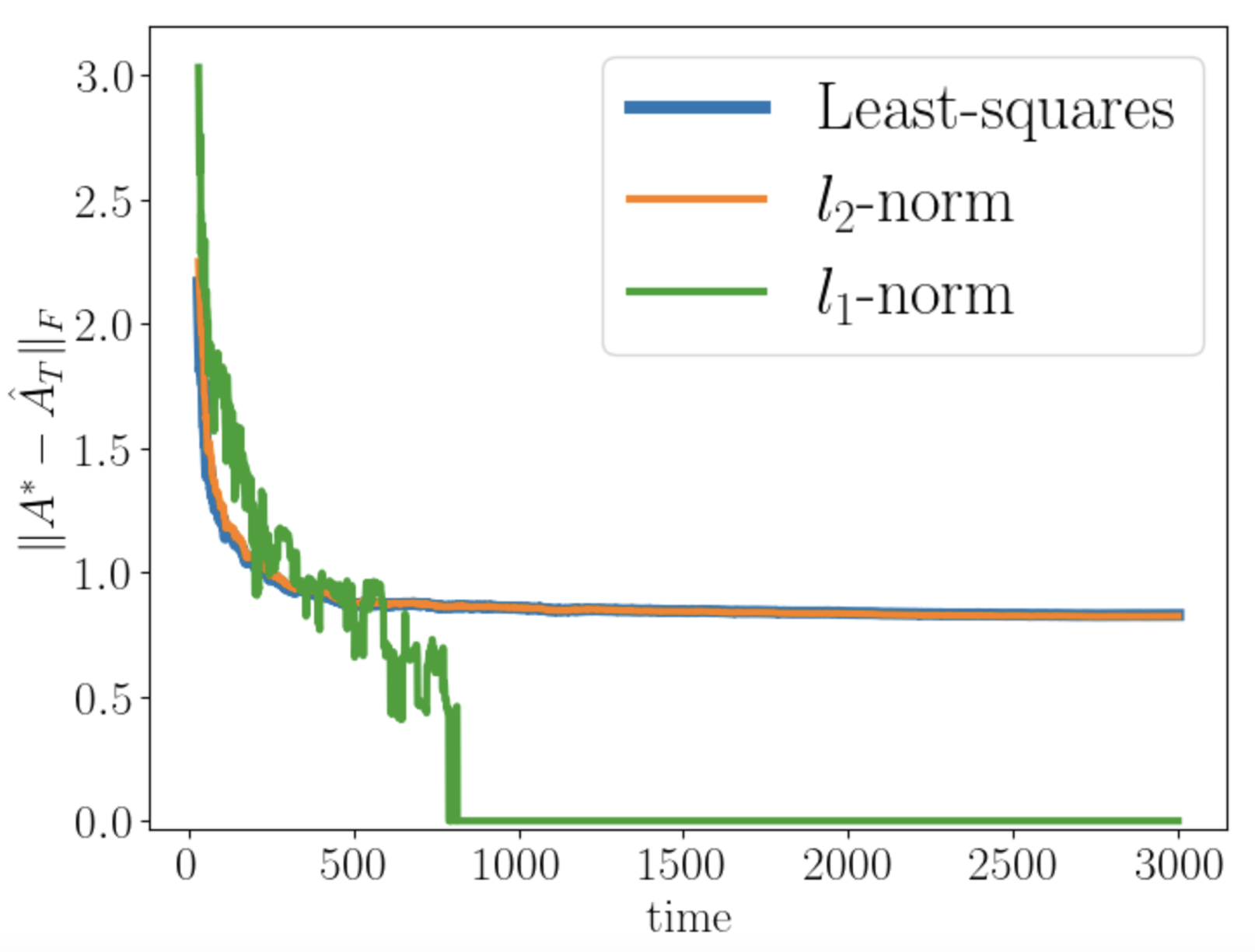}
         \caption{$p=0.45, d=10$}
         \label{p45}
     \end{subfigure}
     \begin{subfigure}[b]{0.235\textwidth}
         \centering
         \includegraphics[width=\textwidth,height=80pt]{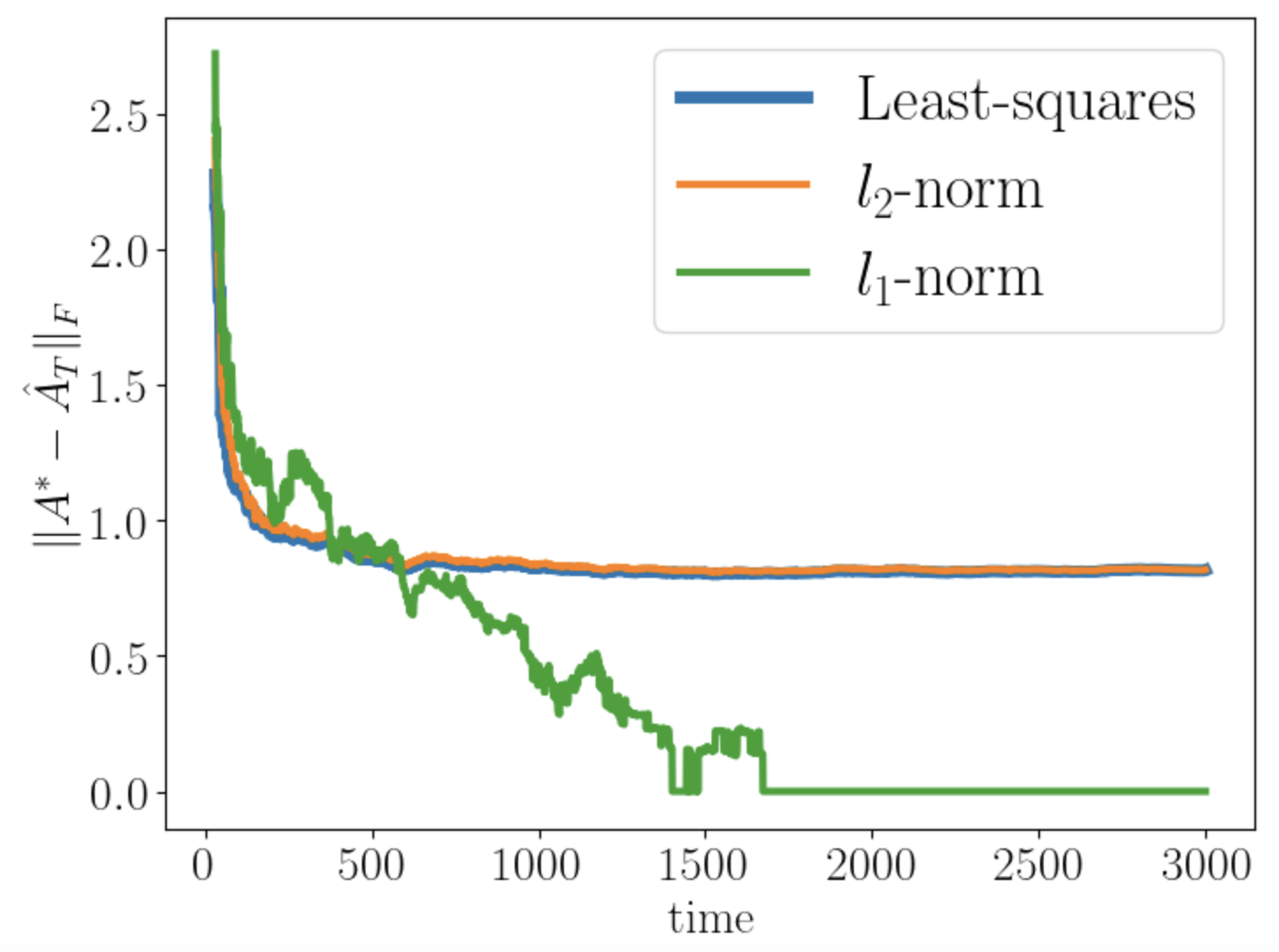}
         \caption{$p=0.47, d=10$}
         \label{p47}
     \end{subfigure}
     \begin{subfigure}[b]{0.235\textwidth}
         \centering
         \includegraphics[width=\textwidth,height=80pt]{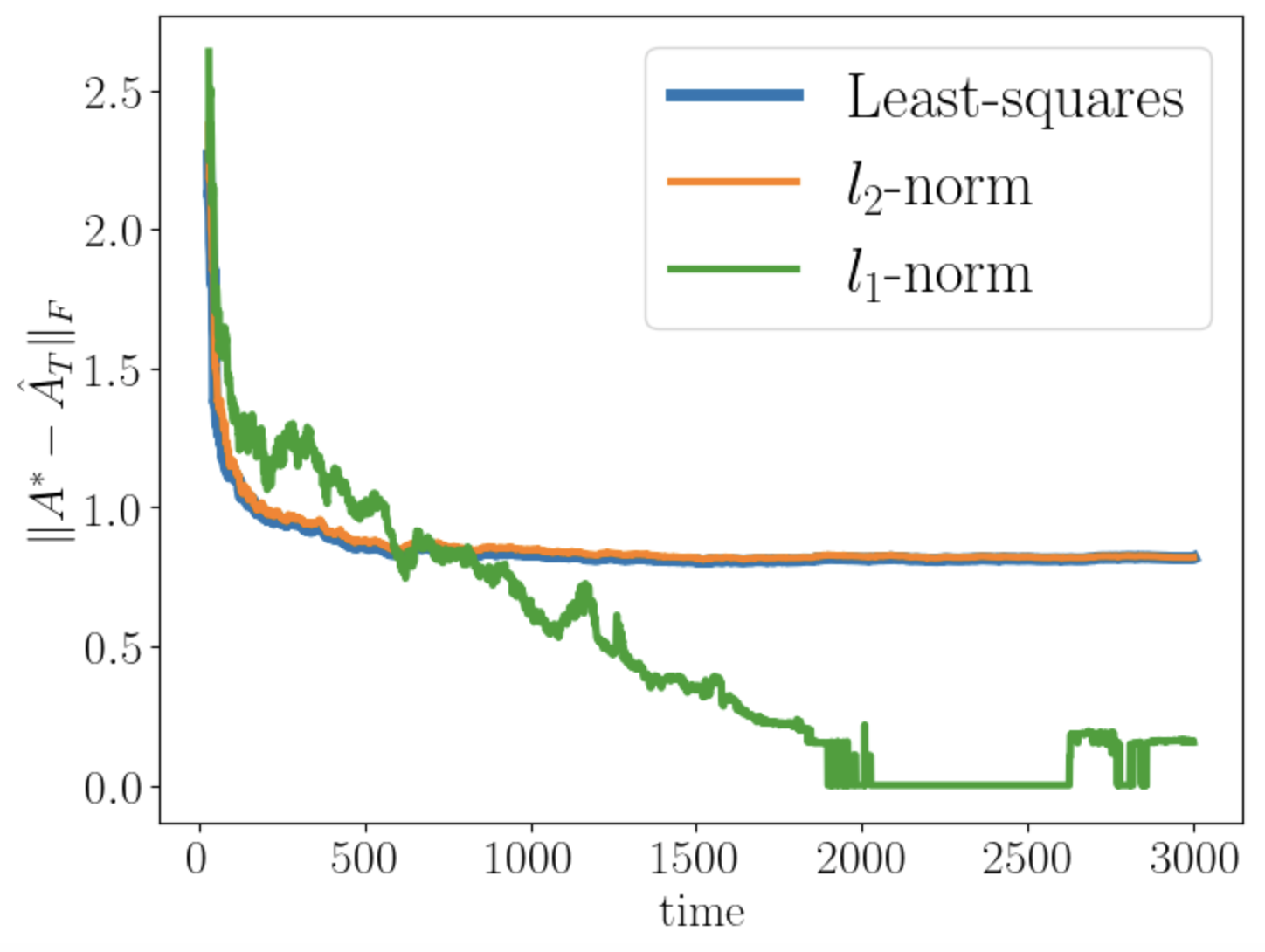}
         \caption{$p=0.48, d=10$}
         \label{p48}
     \end{subfigure}
     \begin{subfigure}[b]{0.235\textwidth}
         \centering
         \includegraphics[width=\textwidth,height=80pt]{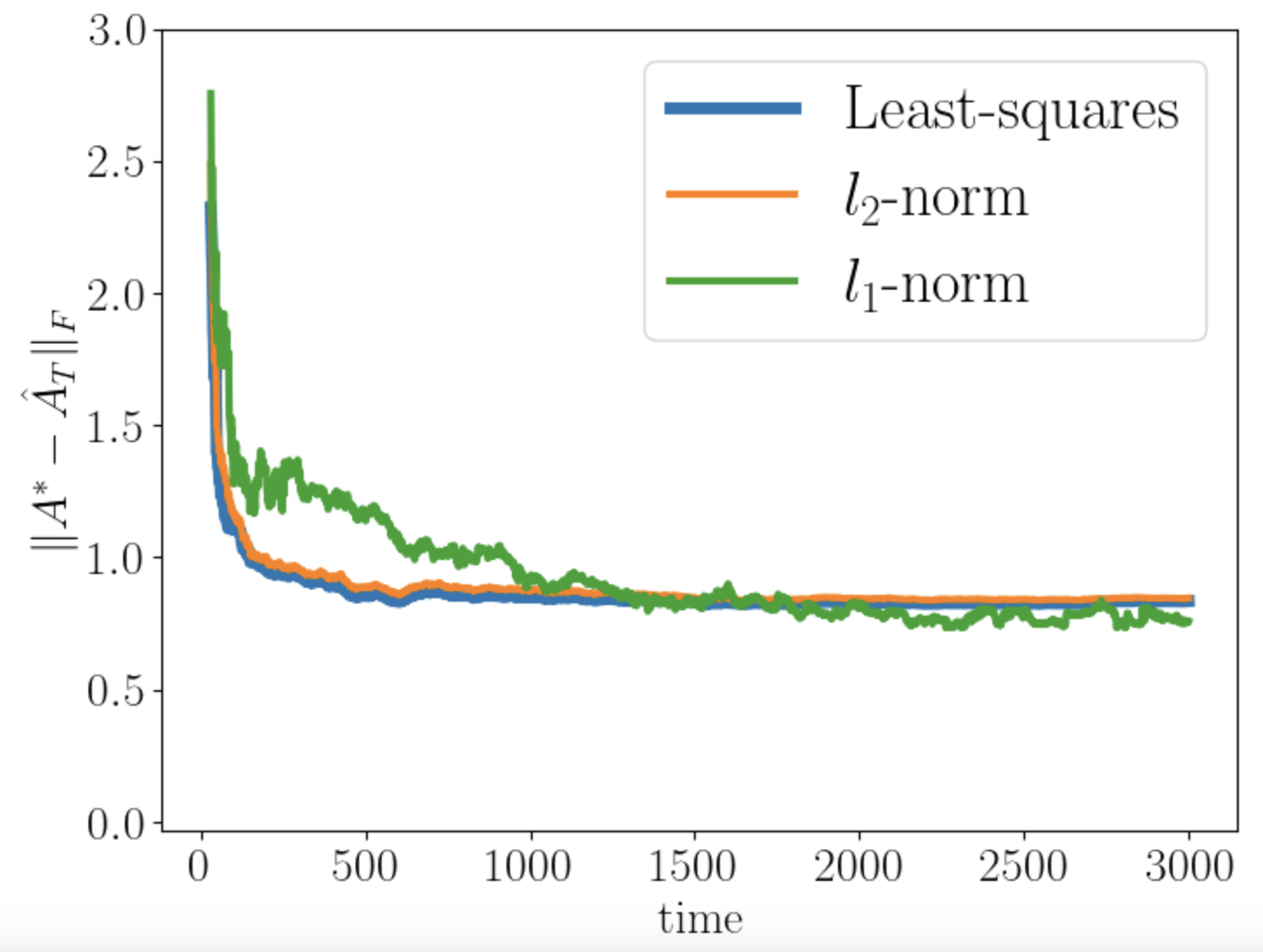}
         \caption{$p=0.5, d=10$}
         \label{p5}
     \end{subfigure}
        \caption{$l_1$-norm estimator vs. other estimators in Example \ref{support3}.}
        \label{ex2}
\end{figure}
\begin{example}\label{support3}
For this example, we generated a random matrix whose operator norm is 0.95. 
We consider the case of completely arbitrary disturbances with the attack probability $p<0.5$, which supports Theorem \ref{pless5}. We choose 
each coordinate $w_t^i$ of the disturbance $w_t$ to be Gaussian with a mean of $100\cdot (\text{sgn}(x_t^i)+2)$ and a variance of 5 when the system is under attack. This adversarial disturbance implies that the attack is almost always positive so as to mislead the estimators towards a positive bias. Figure \ref{ex2} shows that the $l_1$-norm estimator successfully identifies the matrix when $p$ is at most 0.47, while the other estimators again fail to improve beyond a certain level of error. Figures \ref{p45} and \ref{p47} imply that a slight change in $p$ leads to a considerable difference in the exact recovery time. This is because, unlike in Example \ref{support2}, the recovery time is modified to $\tilde{\Theta}\bigr(\frac{d}{p(1-2p)^2}\bigr)$, which is very sensitive near $p=0.5$. Figure \ref{p48} demonstrates this sensitivity as $p$ approaches 0.5, showing that the error alternates between zero and positive values. Finally, when $p$ is exactly 0.5, Figure \ref{p5} indicates that the $l_1$-norm estimator has a plateau in error, similar to that of the other estimators. This aligns with Theorem \ref{pless5}, which only holds for $p<0.5$.
This observation identifies the critical bound of 0.5, below which the $l_1$-norm estimator can ultimately recover the true matrix even in the presence of arbitrarily large noncentral attacks. 
\end{example}

\section{Conclusion} \label{sec: conclusion}
In this paper, we study the capability of the $l_1$-norm estimator to exactly identify the true matrix in the linear system identification problem, where the system suffers from adversarial noncentral disturbances. We show that the true matrix is exactly recovered under a symmetry assumption and the attack probability $p$ being less than 1. Furthermore, if $p<0.5$, the $l_1$-norm estimator prevails against arbitrary adversarial noncentral disturbances and achieves the exact recovery.
This is the first result in the literature showing the possibility of accurate learning of systems under correlated, adversarial, and nonzero-mean disturbances.







\printbibliography

\end{document}